\newcounter{myctr}

\documentclass{ws-acs}
\usepackage{url}

\usepackage{setspace} %
\begin{document}

\makeatletter
\def\@biblabel#1{[#1]}
\makeatother

\markboth{Michael Barnsley, Brendan Harding}{Fractal transformations in 2 and 3 dimensions}

%
\catchline{}{}{}{}{}
%

\title{FRACTAL TRANSFORMATIONS IN 2 AND 3 DIMENSIONS\\
Dedicated to Benoit Mandelbrot
}

\author{\footnotesize MICHAEL BARNSLEY
}

\address{Mathematics Department, Australian National University\\
Canberra, Australian Capital Territory 0200, Australia
\\
mbarnsley@aol.com}

\author{BRENDAN HARDING}

\address{Mathematics Department, Australian National University\\
Canberra, Australian Capital Territory 0200, Australia\\
brendan.harding@anu.edu.au}

\maketitle

\begin{history}
\received{(received date)}
\revised{(revised date)}
\end{history}

\begin{abstract}



We present some work relating to fractal transformations on masked iterated function systems and demonstrate how well known algorithms for generating fractal transformations can be modified for these systems. We also demonstrate that these algorithms work equally well when applied to three dimensional data sets and suggest some possible applications to special effects and modelling.



\end{abstract}


\newpage 

\section*{INTRODUCTION}

Iterated function systems (IFS's) have been used in many applications as an effective way to describe and generate sets. For example it is well known that a complete family of fractal ferns can be generated using only 4 contractive mappings on the Euclidean plane. Sets generated by an iterated function system can be made to appear much more complex by introducing shading and transparency effects by means of plotting the approximate fractal meaasure \cite{bar93:1993:FE}. Further, the idea of colour stealing \cite{bar06:2006:SF} allows for such objects to be coloured in interesting ways. Fractal transformations have been recently introduced \cite{bar09:2009:TSRS} and use the intricate structure generated by IFS's to produce complex mappings between attractor sets. The main idea of fractal transformations is to generate an image on an attractor by colour stealing from another image supported on an attractor with a very similar IFS structure. The result of this transformation is generally not continuous. However, there are known conditions for which these transformations are continuous and even homeomorphisms \cite{bar09:2009:TSRS}. In section 1 we review the key definitions and theorems relevant to this paper including some recently published results extending the theory to masked IFS's \cite{bhi11:2011:TFI}.

The chaos game algorithm \cite{bardem85:1985:IFS} is a fast and well-known algorithm for producing images of attractor sets. However the chaos game takes an undetermined amount of time to produce a complete image, this is particularly relevant when plotting high resolution images. Pixel chaining \cite{sla07:2007:RTIS} is an algorithm which overcomes this issue but can still take some time to complete. In section 2 we review the many algorithms that can be used for generating attractor sets. We discuss the advantages and disadvantages from a computational perspective and describe modifications that allow these methods to be applied to fractal transformations. Further, we describe new ways of applying these algorithms such that they may be used to plot transformations arrising from masked IFS's. We also present ideas for the paralellisation of such algorithms.


In section 3 we review the application of such transformations in the artistic transformation of images \cite{bar09:2009:TSRS}, modelling of objects \cite{demko:construction} and image synthesis \cite{sla07:2007:RTIS} in 2 dimensions. This will lead us into section 4 where we extend these ideas demonstrating ???for the first time??? how they may be applied to three dimensional data sets. We conclude in section 5 with some ideas for future work and improvements.

This paper is written in memory of Benoit Mandelbrot who has given the authors of this paper much inspiration in the pursuit to better describe the world around us using the tools of fractal geometry.

\section{THEORY OF FRACTAL TRANSFORMATIONS}

In this section we will review the definitions and theorems that are necessary to define these fractal transformations. Initially these transformations were described in the context of hyperbolic IFS's on compact metric spaces \cite{bar09:2009:TSRS}. It was also demonstrated how they can be constructed in two dimensions on attractors in which the images of the attractor overlap on a set of measure 0 (two dimensional Lebesgue measure). Recent work \cite{bhi11:2011:TFI} has shown that these transformations may be contructed in a much more general setting which we will review again here.

Let $X$ be a non-empty compact Hausdorff space and let $K(X)$ be the set of nonempty compact subsets of $X$. When endowed with the Vietoris topology this becomes a compact Hausdorff space. Now suppose that $\{f_i :X\rightarrow X\mid i\in I\}$ is a sequence of continuous functions where $I=\{1,2,\dots N\}$ is a finite index set with the discrete topology. Then
\[\mathcal{F}:=(X;f_1,\dots,f_N)\]
is called an \textit{iterated function system} on $X$. We can then define the mapping
\[\Pi:I^\infty \rightarrow K(X),\sigma\mapsto\bigcap_{k=1}^{\infty}{f_{\sigma_1}\circ f_{\sigma_2}\circ\cdots\circ f_{\sigma_k}(X)}\]
for all sequences $\sigma=\sigma_1 \sigma_2 \sigma_3 \dots\in I^\infty$. This is well-defined because $\Pi(\sigma)$  is the intersection of a nested sequence of nonempty compact sets. 
\begin{definition} Let $\mathcal{F}:=(X;f_1,\dots,f_N)$ be an IFS over a compact Hausdorff space X. If $\Pi(\sigma)$ is a singleton for all $\sigma\in I^\infty$ then $\mathcal{F}$ is said to be point-fibred, and the coding map of $\mathcal{F}$ is defined by
\[\pi:I^\infty \rightarrow A,\{\pi(\sigma)\}=\Pi(\sigma)\]
where $A\subset X$ denotes the range of $\pi$.
\end{definition}
These defintions and the results that follow provide a much more general setting than the classical work involving contraction mappings on compact metric spaces. We now introduce a notion of attractor in this setting.
\begin{definition} Let $\mathcal{F}$ be an IFS on a compact Hausdorff space $X$. An attractor of $\mathcal{F}$ is a set $A\in K(X)$ that satisfies the following:
\begin{itemize}
\item[(i)] $\mathcal{F}(A)=A$, and
\item[(ii)] there exists an open set $U\subset X$ such that $A\subset U$ and $lim_{k\rightarrow\infty}{\mathcal{F}^k (B)=A}$ for all $B\in K(U)$. (The limit is with respect to the Vietoris topology on $K(X)$.)
\end{itemize}
We also denote the largest set $\mathcal{B}\subset X$ such that the second condition holds for all $B\in K(\mathcal{B})$ to be the basin of $A$. 
\end{definition}
We summarise some results that arise from these definitions in the following theorem.
\begin{theorem}
If $\mathcal{F}$ is a point-fibred IFS on a compact Hausdorff space $X$ then 
\begin{itemize}
\item[(i)] the coding map $\pi:I^\infty\rightarrow A$ is continuous;
\item[(ii)] $\mathcal{F}:K(X)\rightarrow K(X)$ has a unique fixed-point $A\in K(X)$;
\item[(iii)] $A$ is the unique attractor of $\mathcal{F}$;
\item[(iv)] $A$ is equal to the range of the coding map $\pi$, namely $A=\pi(I^\infty)$ noting that
\[\pi(\sigma)=\lim_{k\rightarrow\infty}{f_{\sigma\mid k}(a)}\text{ where }f_{\sigma\mid k}:=f_{\sigma_1}\circ f_{\sigma_2}\circ\cdots\circ f_{\sigma_k};\]
\item[(v)] the basin of $A$ is $X$;
\item[(vi)] if  $B\in K(X)$ then $\{\pi(\sigma)\}=\lim_{k\rightarrow\infty}{f_{\sigma\mid k}(B)}$ for all $\sigma\in I^\infty$.
\end{itemize}
\end{theorem}
\begin{proof}
See \cite{bhi11:2011:TFI}, Theorem 2.2 and Theorem 2.4.
\end{proof}
This leaves us with one more definition fundamental to the developement of fractal transformations, that is the notion of an address.
\begin{definition} Let $\mathcal{F}$ be an IFS on a compact Hausdorff space $X$. The set $I^\infty$ is called the code space of $\mathcal{F}$ and a point $\sigma\in I^\infty$ is called an address of $\pi(\sigma)\in A$.
\end{definition}

For the remainder of this sections we will refer to $\mathcal{F},\mathcal{G},\mathcal{H}$ as being point-fibred IFS's on the compact Hausdorff space $X$. We will also refer to an IFS as being injective and/or open when the maps that it comprises are injective and/or open respecively.

We now have the essential definitions and theorems necessary to define a generalised theory of fractal transformations. The coding map of an IFS generates an intricate address structure on the attractor and by mapping from one attractor to another by means of the addresses we are able to define fractal transformations. There is one subtle issue that must be overcome to make these transformations well-defined and this relates to our definition of \textit{an} address. It is possible that any given point $a\in A$ may have many addresses, even infinitely many. Originally this was dealt with by defining the notion of a top address for a given point \cite{bar09:2009:TSRS}. We will come back to this later where we present a more general way of picking out a single address for each point $a\in A$. For now we give the following definition.

\begin{definition}
Let $\pi:I^\infty \rightarrow A$ be the coding map of $\mathcal{F}$. $\Omega\subset I^\infty$ is called an address space for $\mathcal{F}$ if $\pi(\Omega)=A$ and $\pi\mid_\Omega :\Omega\rightarrow A\subset X$ is one-to-one. The corresponding map
\[\tau:A\rightarrow\Omega,x\mapsto(\pi\mid_\Omega)^{-1}(x),\]
is called a section of $\pi$.
\end{definition}

A section has many useful properties which we summarise in the following theorem (see \cite{bhi11:2011:TFI} Theorem 3.2).
\begin{theorem} Let $\mathcal{F}=(X;f_1,\dots,f_N)$ be a point-fibred IFS on the compact Hausdorff space $X$ with attractor $A$, code space $I^\infty$ and coding map $\pi:I^\infty \rightarrow A$. If $\tau:A\rightarrow\Omega$ is a section of $\pi$ then
\begin{itemize}
\item[(i)] $\tau:A\rightarrow\Omega$ is bijective;
\item[(ii)] $\tau^{-1}:\Omega\rightarrow A$ is continuous;
\item[(iii)] $\pi\circ\tau=i_A$, the identity map on $A$, and $\tau\circ(\pi\mid_\Omega)=i_\Omega$, the identity map on $\Omega$;
\item[(iv)] if $\mathcal{F}$ is injective and $f_i (A)\cap f_j (A)=\emptyset$ for all $i,j\in I$ with $i\neq j$, then $\Omega=I^\infty$;
\item[(v)] if $\Omega$ is closed then $\tau:A\rightarrow\Omega$ is a homeomorphism;
\item[(vi)] if $A$ is connected and $A$ is not a singleton, then $\tau:A\rightarrow\Omega$ is not continuous.
\end{itemize}
\end{theorem}
\begin{proof}
See \cite{bhi11:2011:TFI}, Theorem 3.2.
\end{proof}

Now we define a fractal transformation.

\begin{definition} Let $\mathcal{F}=(X;f_1,\dots,f_N)$ and $\mathcal{G}=(Y;g_1,\dots,g_N)$ be a point-fibred IFS's on the compact Hausdorff spaces $X$ and $Y$ respectively. Let $A_\mathcal{F} \subset X$ be the attractor of $\mathcal{F}$ and $\pi_\mathcal{F} :I^\infty \rightarrow A_\mathcal{F}$ be the coding map of $\mathcal{F}$. Similarly let $A_\mathcal{G} \subset Y$ be the attractor of $\mathcal{G}$ and $\pi_\mathcal{G} :I^\infty \rightarrow A_\mathcal{G}$ be the coding map of $\mathcal{G}$. If $\tau_\mathcal{F}:A_\mathcal{F} \rightarrow \Omega_\mathcal{F}\subset I^\infty$ is a section of $pi_\mathcal{F}$ then the corresponding fractal transformation is defined to be 
\[T_\mathcal{FG}:A_\mathcal{F} \rightarrow A_\mathcal{G} ,x\mapsto\pi_\mathcal{G} \circ\tau_\mathcal{F} (x).\]
\end{definition}

The following theorem describes some coontinuity properties of fractal transformations and is of particularly important as it makes these transformations interesting for the applications we will later describe.

\begin{theorem} Let $\mathcal{F}$ and $\mathcal{G}$ be defined as in the previous definition with $T_\mathcal{FG}:A_\mathcal{F} \rightarrow A_\mathcal{G}$ being the corresponding fractal transformation.
\begin{itemize}
\item[(i)] If, whenever $\sigma,\omega\in\overline{\Omega_\mathcal{F}}$, $\pi_\mathcal{F}(\sigma)=\pi_\mathcal{F}(\omega)\Rightarrow\pi_\mathcal{G}(\sigma)=\pi_\mathcal{G}(\omega)$, then $T_\mathcal{FG}$ is continuous.
\item[(ii)] If $\Omega_\mathcal{G}:=\Omega_\mathcal{F}$ is an address space for $\mathcal{G}$, and if, whenever $\sigma,\omega\in\overline{\Omega_\mathcal{F}}$, $\pi_\mathcal{F}(\sigma)=\pi_\mathcal{F}(\omega)\Leftrightarrow\pi_\mathcal{G}(\sigma)=\pi_\mathcal{G}(\omega)$, then $T_\mathcal{FG}$ is a homeomorphism and $T_\mathcal{GF}=T_\mathcal{FG}^{-1}$.
\end{itemize}
\end{theorem}
\begin{proof}
See \cite{bhi11:2011:TFI}, Theorem 3.4.
\end{proof}

This provides an up-to-date definition of fractal transformations. From a computational perspective such a general viewpoint is not necessary, but this general setting does include the earlier work involving contraction mappings on compact metric spaces which is less abstract to compute. If, at this point we tried to construct some examples of these transformations we would soon run into a problem. We got around the issue of having multiple addresses for points on the attractor by introducing the address space, but given an IFS $\mathcal{F}$ we gave no notion of how one may construct an address space. In \cite{bar09:2009:TSRS} the notion of a tops code space was introduced. The idea is do impose a lexicographic ordering on $I^\infty$ and construct your address space by looking at each point on the attractor and picking the largest of its addresses with respect to this ordering. This is somewhat limited in the possible address spaces and sections it can produce. The idea of masks gives us a much more general way of constructing address spaces.

\subsection{Masks}

\begin{definition} Let $\mathcal{F}$ be a point-fibred IFS on the compact Hausdorff space. Let $A$ be the attractor of $\mathcal{F}$. A finite sequence of sets $\mathcal{M}:=\{M_i \subset A\mid i\in I\}$ is called a mask for $\mathcal{F}$ if 
\begin{itemize}
\item[(i)] $M_i \subseteq f_i (A)$, for each $i\in I$;
\item[(ii)] $M_i \cap M_j =\emptyset$, for all $i,j\in I$ with $i\neq j$;
\item[(iii)] $\cup_{i\in I}{M_i}=A$.
\end{itemize}
\end{definition}

A mask is essentially a partition of the attractor of an IFS. In particular it means that for any $x\in A$ there exists a unique $i\in I$ such that $x\in M_i \subseteq f_i(A)$. This property enables us to define a dynamical system on the attractor which is associated with the mask.

\begin{definition} Let $\{M_i :i\in I\}$ be a mask for an injective point-fibred IFS $\mathcal{F}$ with attractor $A$. The associated masked dynamical system is defined as
\[T:a\rightarrow A, x\mapsto\left\{ \begin{array}{l} f_{1}^{-1}(x), \; x\in M_1, \\ f_{2}^{-1}(x), \; x\in M_2, \\
 \;\; \vdots \\ f_{N}^{-1}(x), \; x\in M_N \end{array} \right. \]
\end{definition}

This dynamical system can be thought of as an address generating dynamical system. The idea behind this is that if you follow the orbit of each $x\in A$ under this dynamical system then you can generate a unique address simply by recording which mask the orbit lands in after each iteration. This is made more precise in the following theorem.

\begin{theorem} Let $\mathcal{F}$ be an injective point-fibred IFS with attractor $A$. Let $\mathcal{M}:=\{M_i : i\in I\}$ be a mask for $\mathcal{F}$ with the associated dynamical system $T:A\rightarrow A$. Let $x\in A$ and $\{x_n\}_{n=0}^{\infty}$ be the orbit of $x$ unser $T$; that is, $x_0 =x$ and $x_n =T^n (x_0 )$ for $n=1,2,\dots$. Let $\sigma_k (x)\in I$ be the unique symbol such that $x_{k-1} \in M_{\sigma_k}$, for $k=1,2,\dots$. Then
\[\Omega_\mathcal{M} =\{\sigma\in I^\infty \mid\sigma:=\sigma_1 (x)\sigma_2 (x)\sigma_3 (x)\dots\in I^\infty ,x\in A\}\]
is an address space for $\mathcal{F}$.
\end{theorem}
\begin{proof}
See \cite{bhi11:2011:TFI}, Theorem 4.3.
\end{proof}

\begin{definition}
Let $\mathcal{F}$ be an injective point-fibred IFS with an address space $\Omega_\mathcal{M}\subset I^\infty$ as provided in the previous theorem. $\Omega_\mathcal{M}$ is calles a masked address space for $\mathcal{F}$ and the corresponding section of $\pi$, say $\tau:A\rightarrow\Omega_\mathcal{M}$, is called a masked section of $\pi$.
\end{definition}

At this point we note that the tops address space can be generated by masks. If we were choose $M_1=f_1 (A)$ and $M_i=f_i (A)\backslash{\cup_{k=1}^{i-1}{f_k (A)}}$ for all $i\in I\backslash\{1\}$ then the corresponding masked address space would be the tops address space. Masked address spaces and masked sections have all of the properties of address spaces and sections previously given. They also have some additional properties which we include in the following theorem.
 
\begin{theorem} If $\mathcal{F}=(X;f_1,\dots,f_N)$ is an injective point-fibred IFS on the compact with attractor $A$, code space $I^\infty$, coding map $\pi:I^\infty \rightarrow A$, mask $\mathcal{M}:=\{M_i : i\in I\}$, masked address space $\Omega_\mathcal{M}$ and masked section $\tau:A\rightarrow\Omega_\mathcal{M}$, then
\begin{itemize}
\item[(i)] if $\mathcal{F}$ is open then $\tau:A\rightarrow\Omega_\mathcal{M}$ is continuous at $x\in A$ iff $T^{k-1}(x)\in \text{Int}_A (M_{\tau(x)_k})$ for all $k=1,2,3,dots$;
\item[(ii)] the shift map $S:\Omega_\mathcal{M} \rightarrow \Omega_\mathcal{M},\sigma_1 \sigma_2 \sigma_3 \dots\mapsto\sigma_2 \sigma_3 \sigma_4 \dots$ is well-defined, with $S(\Omega_\mathcal{M})\subset\Omega_\mathcal{M}$;
\item[(iii)] the following diagram commutes
(need to add commutative diagram)
\item[(iv)] if there is $i\in I$ such that $M_i=f_i(A$ then $S(\Omega_\mathcal{M})=\Omega_\mathcal{M}$.
\end{itemize}
\end{theorem}
\begin{proof}
See \cite{bhi11:2011:TFI}, Theorem 4.5.
\end{proof}

One last property of masks worth mentioning relates to the variety of address spaces it is possible to generate using masks. Theorem 4.7 of \cite{bhi11:2011:TFI} tells us that if two different masks on a point-fibred IFS $\mathcal{F}$ differ by more than a set of measure 0 (with respect to some measure $\mu$ on $A$) then the two masked sections are different for $\mu$-almost all $x\in A$. Hence we can see that masks give us a much more general way of generating address spaces then the tops method and hence can lead to a much greater variety of fractal transformations.

Masks are also a great tool from a computational perspective as it is easy to apply the masked dynamical system in a finite loop to approximate a unique address for data elements we may wish to transform.

\section{ALGORITHMS}

We will descibe algorithms in the context of transformaing images. With this in mind it is simplest to consider hyperbolic IFS's on the unit square, i.e. $[0,1]^2 \subset \mathbb{R} ^2$. In this space the most natural metric to work with is the Euclidean metric and there has been lot of theory in the literature which works with IFS's on compact metric spaces \cite{bar09:2009:TSRS,abvw10:2010:TP,bar93:1993:FE,hut81:1981:FSS}. It is simplest to consider transforming a whole image which we embed onto the unit square. Of course it is possible to apply fractal transformations on much more complicated fractal sets which could be done with modifications of the algorithms described here. 

Defining IFS's such that the attractors are the unit square can be made a simple task by making use of the collage theorom \cite{bar93:1993:FE}. This theorem tells us that you can approximate the desired attractor by defining contractive mappings that cover the desired attractor with smaller copies of itself. The difficult part in most cases is defining the IFS's and masks appropriately such that the transformation is continuous. We can do this relatively easily with affine transformations by ensuring that their ranges overlap only on their boundaries or, more precisely, on a set of measure 0 with respect to two-dimensional Lebesgue measure. For example take the IFS's 
\begin{equation} \mathcal{F}= \left\{ [0,1]^2:\begin{array}{l} f_{1}(x,y)=(x/2,y/2), \\ f_{2}(x,y)=((x+1)/2,y/2),  \\
  f_{3}(x,y)=((x+1)/2,(y+1)/2), \\ f_{4}(x,y)=(x/2,(y+1)/2) \end{array} \right\}  \end{equation}
\begin{equation}  \mathcal{G}= \left\{ [0,1]^2: \begin{array}{l} g_{1}(x,y)=(ax,by), \\ g_{2}(x,y)=(a+(1-a)x,by),   \\
  g_{3}(x,y)=(a+(1-a)x,b+(1-b)y), \\ g_{4}(x,y)=(ax,b+(1-b)y) \end{array}  \right\} \end{equation}
where $a,b\in(0,1)$. As you can see the ranges of the functions on each IFS only overlap on boundaries. Also since the basic structure of the functions in the two IFS's is the same (by which we mean the similar relative position of the ranges of the functions and the same orientation of the function pairs) then it can be easily shown that the two IFS's have the same tops code space and hence it is possible to define a fractal transformation between these IFS's which will be a homeomorphism.

There are many algorithms for producing images of these transformations but first we will make some general notes on the implementation. The first comment is that since we are in digital space we need only transform a finite number of points on the attractor. Depending on the properties of the fractal transformation we intend to apply and given the chaotic nature of the underlying dynamics we may want to think carefully about the exact way in which we choose a point for a given data element. In general, if we are applying a continuous fractal transformation that is relatively subtle then this isn't too much of a concern. The second remark regards the fact that points on the attractor, and hence each pixel, correspond to addresses which are infinitely long entities in the address space. Since we cannot compute infinite codes we must determine if a large finite code will suffice and how large the finite code should be to have a reasonable error in the calculation. Obviously a smaller code would require less memory and the run time of the algorithm will be reduced but the length must be long enough such that the transformation is applied accurately. Hence we must try and balance the two possibly making a compromise on one or the other depending on run time constraints. The length of the code which results in a visually accurate transformation relies mostly on three factors, the size of the dataset to be transformed, the rate of convergence of the functions in one IFS and how difficult it is to describe the inverse functions in the other IFS. For a visually accurate transformation we may ask that the addresses are long enough such that every address of that length corresponds to a sequence of functions that converges to a point with an accumulated error of no more than the width of a data element (or distance between two data elements). One way to determine this is to look at the contraction factors of the functions in each IFS. The length of the address needs to be large enough such that iterating the function with the largest contraction factor for that many iterations gives a result accurate to within half a pixel width. I.e. If $M$ is the length of the code, $c$ is the largest contraction factor and $\epsilon$ is the distance between data elements, then it is required that $M$ is the smallest integer such that
\[ \frac{c^{M+1}}{1-c} < \frac{\epsilon}{2} . \]
The idea behind this inequality is that tail of the geometric series $\sum_{i=0}^{\infty}{c^i}$ is smaller than the maximum desired error. The convergence within half of $\epsilon$ is to ensure the result is more accurate when rounding off to the nearest data element on either side but it could be relaxed a little to be within $\epsilon$ without affecting the output by a noticeable amount. Another error to consider is the error that can accumulate when applying the masked dynamical system to generate codes. An 8 byte double has 15 digits so we start with a rounding error of $\delta=10^{-15}$. Since the masked dynamical system is an expanding system this error will rapidly expand as we continue to iterate. The rate it increases is dependant upon the largest exanding factor of the functions in the expanding system (i.e. invert the smallest contraction factor of the IFS). If we call this expanding factor $d$ and we assume on each iteration there is an expansion of this factor plus an additional rounding error then the error accumulates by $\delta\mapsto d*\delta+10^{-15}$. For example if $d=4$ then after $20$ iterations the error accumulates to almost $1.47*10^-3$ at which point we begin to become less confident that we are following the correct orbit.

Rather than determine address length before hand it is also possible that the determination of appropriate code length can be done dynamically. For example, suppose most of the maps have a much smaller contraction factor than the largest, then the code length will not need be as long for most of the pixels to produce an accurate result. In this setting appropriate code lengths are estimated as each symbol in the code is calculated and in that way no more calculations are done than necessary.

Now we will describe some algorithms and discuss some of their advantages and disadvantages.

\subsection{Chaos Game Algorithm}
One of the quickest algorithms to develop a rough image of the resulting transformation applied to an image is by playing the chaos game. This was described by Barnsley and Demko \cite{bardem85:1985:IFS} and is the most common method for plotting attractors of IFS's.

The chaos game consists of randomly iterating the functions in an IFS and plotting the result at each step. As the functions are randomly iterated the resulting points will jump to different locations on the attractor of the IFS in a seemingly chaotic manner. It was shown by Barnsley and Vince that given enough time the chaos game will eventually hit every point on the attractor \cite{barvin10:2010:CG}. It also has the advantage of working in quite general circumstances where there may not be a contractive metric. Where two IFS's are involved for the purposes of plotting a fractal transformation, the forward iterations of the two IFS's are performed simultaneously according to the pairing of functions between the IFS's with the same index. Upon each iteration, the colour from the point landed on by one of the IFS's is plotted onto the point given by the other IFS. This is the basic notion of 'colour stealing'.

Below is a basic outline of implementation for colour stealing with the chaos game in C: \\
\#include$\langle$random$\rangle$ //this defines the pseudo random number generator rand() \\
int i; //our loop variable for counting iterations \\
int itmax = \textit{enter number here}; // the maximum number of iterations \\
double x1, y1; //colour stealing coordinates \\
double x2, y2; //mapping coordinates \\
for(i,i$<$itmax,i++) \\
\{ \\
ran $=$ rand()\%$N$; \\
if(ran$==0$) \{\textit{apply $f_1$ to (x1,y1), apply $g_1$ to (x2,y2) (setting (x1,y1) and (x2,y2) to be the new results respectively)}\}; \\
\vdots \\
if(ran$==N-1$) \{\textit{apply $f_N$ to (x1,y1), apply $g_N$ to (x2,y2)}\}; \\ 
//now take colour from the point (x1,y1) and apply to the point (x2,y2); \\
\}

In practice the use of the pseudo random number generator here is not ideal and should be replaced with something more reliable. For some values of $N$, rand()\%$N$ has a period which is insufficient for the chaos game to fill an entire image. In the example above each random number is chosen with equal probability but in most circumstances it is more efficient if the probability of choosing a particular pair of functions is proportional to the area given by the range of the function acting on the target image (e.g. functions from the IFS $\mathcal{G}$ in the example above).

The chaos game has one notable advantage, it is fast, at least initially. It is quite a simple algorithm and as a result each iteration takes very little time so it is possible to complete millions of iterations every second. Of course, this does depend on the number of functions in each IFS and how easy or difficult those functions may be to calculate at each iteration, but these factors will affect any algorithm for performing a fractal transformation to an image. The chaos game can therefore fill a good portion of the target pixels quite quickly. However, as an image begins to fill we see the main disadvantage of the chaos game. Because it randomly hits points on the attractor it takes an undetermined number of iterarions to complete an image. Once most of an image has been completed the chaos game will continue to iterate over pixels already filled and there is no way of knowing when it may hit the remaining incomplete pixels. A solution to this problem may be to run the chaos game through a DeBruijn sequence [reference] of sufficient length rather than use a pseudo random number generator to ensure that every pixel is hit. The best way to do this would be to have library of polynomials which generate various length DeBruijn sequences for different bases and then choose an appropriate one based on the required address length. One problem with this is that if one or more of the functions has a large contraction ratio then the length of the required DeBruijn sequence could be extremely large making this a very inefficient method. Another solution is to combine the chaos game with other algorithms, that is run the chaos game for a short time and then use another algorithm to fill the remaining pixels. Such algorithms will be discussed soon. 

Another advantage of the chaos game algorithm is that there is less need to worry about address lengths to minimise errors. This really only comes into play during initial iterations that we don't plot whilst the chaos game converges upon the attractor. Any other numerical errors that occur are being continuously reduced due to the application of contraction mappings. The chaos game also has the advantage of requiring minimal memory to run, the bulk of this consists of one buffer to store the image to be transformed and a second buffer to plot the transformed image onto. Another disadvantage of the chaos game is that it cannot easily produce the image of a fractal transformation where masks are being used. A modification of the chaos game can be made to accurately plot transformations of masked IFS's but it is more efficient to use alternative algorithms in these cases. The modification requires keeping track of how close the chaos game orbit is to the reverse of an orbit arising from the masked dynamical system and only plotting points when this distance is sufficiently small. As a result you may hit a given data element many times before satisfying this condition which can draw the algorithm out for quite some time.

\subsection{Perpixel algorithm}
The perpixel method is applied as follows. For each pixel we find an address corresponding to that pixel by reverse iterating the correspinding point $x$ by successively applying the appropriate inverse of functions from the IFS $\mathcal{G}$. If the IFS has a mask then this is equivalent to running the masked dynamical system $T$ for the given point. This is done for finitely many iterations corresponding to the desired address length. In the case of finding the tops address for each point it is also possible do use the following method:

First apply $g_{1}^{-1}$ to the pixel co-ordinates. If the result lies in $[0,1]^2$ then we keep the result and assign $1$ as the first code in the sequence. If not then we try $g_{2}^{-1}$ and so on. At least one of the $N$ inverse functions must work assuming the attractor of the IFS is $[0,1]^2$. We then take the resulting point from this procedure and repeat the process to find out the next code in the sequence and continue to repeat until an appropriate address length, $k$, has been reached (where the determination of this length has been previously discussed).  

Once the address of this point has been determined ($\sigma\mid_k =\tau_\mathcal{G} (x)$) it can then be applied to the functions in $\mathcal{F}$ for some initial point on $A_\mathcal{F}$, that is we apply coding map, $\pi_\mathcal{F}$, to the finite address sequence. The resulting point $y=\pi_\mathcal{F} (\sigma\mid_k)$ corresponds to a pixel on the picture from which we take the colours and apply to our starting pixel on the target image. After performing this calculation for every pixel we will clearly have a complete image of the transformation.

This method is slower than the chaos game to begin with since there are many more calculations required for each pixel. Where the chaos game only needs to apply one function from each IFS in order to plot a point, this algorithm must apply a function from each IFS a number of times corresponding to the desired address length $k$. However as an image completes and the chaos game tends to repeatedly go over points it has already filled the perpixel algorithm will continue to steadily fill pixels and as a result the perpixel method can be quicker to complete an image. It is also more reliabe in terms of the runtime being consistent each time the algorithm is run. Speed of the algorithm can also be improved by dynamically determining the appropriate address length. The perpixel method also has the advantage that it works very well with masked IFS's. It also only slightly more memory than the chaos game to run, this would correspond to a small array for keeping track of finite addresses.

Below we have a basic outline of the calculation for a single pixel being transformed using masked IFS's: \\
double x1, y1; //these are the colour stealing co-ordinates \\
double x2, y2; //these are the mapping coordinates \\
x2 = \textit{x coordinate of target pixel}; \\
y2 = \textit{y coordinate of target pixel}; \\
int m; \\
int codelength = \textit{the desired length of addresses}; \\
int code[codelength]; // to record addresses \\
for (m$=0$,m$<$codelength,m$++$) // this loop runs the masked dybamical system\\
\{ \\
if (\textit{(x2,y2) is in $M_1$}) \{\textit{apply } $g_{1}^{-1}$ \textit{ to (x2,y2) (setting (x2,y2) to be the new result)}; code[m]$=0$\}; \\
\vdots \\
if (\textit{(x2,y2) is in $M_N$}) \{\textit{apply } $g_{N}^{-1}$ \textit{ to (x2,y2)}; code[m]$=N-1$\}; \\
\} \\
x1 = 0; \\
y1 = 0; \\
for ((m$=codelength-1$,m$>-1$,m$--$)) // this loop applies the coding map\\
\{ \\
if (code[m]$==0$) \{\textit{apply } $f_{1}$ \textit{ to (x1,y1) (setting (x1,y1) to be the new result)}\}; \\
\vdots \\
if (code[m]$==N-1$) \{\textit{apply } $f_{N}$ \textit{ to (x1,y1)}\}; \\
\} \\
//find the pixel on the image corresponding to (x1,y1), \\
//take the colour data and apply it to the target pixel coordinates \\

\subsection{Other algorithms}
Here we outline some other applications that may be used.

\subsubsection{Pixel Chaining}

Pixel chaining is another algorithm which is somewhat similar to the perpixel method. The basic idea of reverse iterating on one IFS and then forward iterating on the other is the same except for this algorithm we do the pixels in chains rather than one at a time. This algorithm is described by \cite{lu97:1997:FI}. It can be implemented as a modification of the perpixel method described above. For masked IFS's it involves applying the masked dynamical system on an image (for example) until you hit a pixel that you have landed on previously in the orbit, this forms a chain on $A_\mathcal{G}$. At this point we continue to iterate for some finitely many more steps $k$. This gives us one long address which we can iterate forwards on the second IFS $\mathcal{F}$. After the first $k$ iterations we record where we land so that the points in this chain in $\mathcal{F}$ can be mapped to the chain in $A_\mathcal{G}$. The idea is that we take some starting pixel in the target image and then reverse iterate the functions from $\mathcal{G}$. After this chain is completed we pick another point in $\mathcal{G}$ that hasn't been hit and form another chain. Note that the subsequenct chains may intersect previous chains at which point we iterate a further $k$ times before applying the address on the IFS $\mathcal{F}$. 

The pixel chaining algorithm fills in most of the image relatively quickly to start with but slows down towards the end as chains become smaller. At it's slowest it fills pixels no slower than the perpixel method. Hance the pixel chaining method has the advantage of being quicker than the perpixel method. However it is also a little more difficult to implement and requires more memory to run (the extra memory is the result of an extra buffer to store long addresses corresponding to the chains).

\subsubsection{Combinations of algorithms}

One of the fastest ways of completing the fractal transformation of an image consists of combining the chaos game and the perpixel algorithms. We start with the chaos game and as the image begins to fill out the chaos game will start to continually hit pixels that have already been calculated. At this point the perpixel algorithm is implemented to fill in remaining pixels. The complicated part of combining the two algorithms is determining the optimal moment to switch between the two. Optimally we want to switch to the perpixel algorithm as soon as the chaos game starts filling pixels at at slower rate than the perpixel method would. The easiest way is to implement the switch is to simply run the chaos game for a set number of iterations. Experimentally a good number of iterations in many circumstances seems to be $N$ times the total number of pixels in the target image. Alternatively this could be done dynamically. By monitoring how many times the chaos game hits pixels that have been filled before hitting new ones we can program the shift in algorithm to occur when this number is such that the perpixel algorithm is faster. Whilst giving us something closer to the optimum shift the continuous monitoring may slow the chaos game part slightly. 

The pixel chaining algorithm can also be combined with the chaos game but generally it is slightly quicker when the perpixel method is used in this scenario. This is because if the chaos game fills some initial pixels it prevents the pixel chaining algorithm to build the long chains that make it effective. The extra overhead involved in combining the two also acts to slow the chaos game part down a little. The reason for this is that the pixels already filled by the chaos game prevent the pixel chaining algorithm from building up the long chains that make it fast. It also becomes much more difficult to determine the optimal switching point. A combination of the chaos game and perpixel method also requires less memory to run.

\subsubsection{Approximation}

The approximation algorithm is used where we have already performed a fractal transformation of an image but may want to slightly change the parameters of the underlying IFS's to produce a slightly different fractal transformation. In fact it is not necessary to already have a fractal transformation applies as you could start with the identity fractal transformation. The idea is that you simply recalculate the first element of the code corresponding to each image and output the result. By repeating this you get successive approximations of the new transformation. After applying a number of approximations corresponding to the desired code length $k$, the resulting image should become an accurate image of the new transformation, however if the change in parameters of the IFS's is only small, then much less approximations may be required for an accurate result. What is actually going on in this process is that we are moving around on a superfractal (see \cite{bar06:2006:SF}). You may think of a superfractal as a space of fractal transformations and when successively applying the approximation algorithm we are starting from our current point in this space (corresponding to the current fractal transformation) and converging upon a new point which is the fractal transformation given by the new parameters.

The algorithm is applied to masked IFS's as follows, it starts when we generate the first transformation. The first transformation can be produced using any previous algorithm but in order to be able to apply the approximation algorithm it is necessary to record the pixel coordinates on the picture from which we stole the colour for the target image for each pixel on the target image. That is, we have a buffer for the inverse fractal transformation. To apply an approximation, start with a point $x$ in $A_\mathcal{F}$ and then determine which mask $M_i$ of the new IFS this point lies in and apply the corresponding inverse function $f_{i}^{-1}(x)$ keeping a record of the appropriate code element $\sigma_0=i$. The inverse fractal transformation is then applied to the resulting point to obtain $T_{\mathcal{FG}}^{-1}(f_{\sigma_0}^{-1}(x))$. This point is then mapped by the function in $\mathcal{G}$ with index $\sigma_0$, that is we apply $g_{\sigma_0}$ to obtain the point $y=g_{\sigma_0} (T_{\mathcal{FG}}^{-1}(f_{\sigma_0}^{-1}(x)))$. The colour of the resulting pixel at $y$ in $A_\mathcal{G}$ is then applied to our starting pixel at $x$ in $A_\mathcal{F}$. By doing this for every pixel on the target image we have the first approximation. As this is being completed we also record where each of these new pixels were mapped from on the picture to get a new inverse transformation which can be used to repeat the process resulting in successive approximations, i.e. we have a buffer for the inverse fractal transformation that arises from the first approximation. The reason why this works may not be obvious at first but effectively what is occuring here is that at the first approximation we correct the first element of the address for each point, in the second approximation we correct the second address element, and etc.

The major advantage of this algorithm is that the user can very quickly see updated approximations, although some of these approximations don't always look nice to start with. A disadvantage is that by the time an accurate image of the new transformation has been generated by the approximations the same image generally could have already been generated by one of the other faster algorithms. The use of extra buffers also means this algorithm uses more memory than the others. It is possible to make it even faster by using more buffers to record the preimages and images if points in $A_\mathcal{F}$ and $A_\mathcal{G}$ respectively, but the extra buffers requires more memory.

\subsection{Parallel algorithms}

Many of these algorithms have the potential to be parallelised to reduce the time it takes to transform a given dataset. We start by looking at the the chaos game. This can done by simply running the chaos game on multiple threads making sure that the random number generator has a different seed on each thread and that the random number sequences don't overlap on any of the threads. This will fill in lots of points quickly to begin with but each thread will eventually have a lot of overlap. Another method would be to force different sequences on each thread, however this could be quite difficult requiring a lot of communication between thrads. Yet abither suggestion is to use DeBruijn sequences again, the sequence can be split into smaller sections of equal length and distributed across the threads.

The perpixel method has the most potential for parallelising. This can be done by simply dividing the dataset between the threads and letting each thread run the perpixel algorithm for the assigned data elements. There could be small issues with distributing workload here but this can be fixed by distributing the pixels dynamically, i.e. as each thread completes the calculation for a data element you give it another. As a result the  number of threads almost directly corrleates to a decrease in calculation time as $1/(\text{no.of threads})$, with some small differences made by the additional overhead.

The pixel chaining algorithm is not particularly suitable for running on multiple threads. The dataset could be divided up much like the perpixel algorithm but the data element being calculated by one thread could end up in the chain of another another thread making it a redundant calculation. With some additional overhead it would be possible have threads share information so that when one threads chain intersects a chain on another thread then one of the two threads can take over linking the two chains whilst the other thread starts again at a different point. Given the additional overhead and the complexity it generally seems that this isn't the best approach to take.

The approximation algorithm can also be easily run on multiple threads by again dividing up the pixels between the threads. The number of threads used for this algorithm then affects the run-time in a similar manner to the perpixel method.

Algorithms can still be combined when running parrallel threads. The chaos game and the perpixel method can be combined in the same way as before, but here determining the optimal moment to switch algorithms becomes a little more complicated. If done correctly this makes for an extremely fast algorithm compared to much simpler methods.

\subsection{Further Remarks}
Could include some comments about techniques when transforming between sets with non-integer Hausdorff dimension, i.e. comment on the complications that can arise with the masked dynamical system $T$ running off outside the attractor

\section{SOME APPLICATIONS IN 2D}

Here we review the applications suggested in \cite{bhi11:2011:TFI} as well as suggest some more.

\subsubsection{Transformations of images}

The implementation of this is a simple application of one of the algorithms descrived above. We have developed a few programs allows for some simple families of fractal homeomorphisms to be applied to images interactively in real-time as a user controls various nodes that manipulate the underlying IFS's. Some of the programs which have quite a few nodes demonstrate how parameters can be subtley manipulated for the purposes of image beautification. On the other hand, extreme placement of nodes can result in very sharp effects that present new artisctic textures. 

There are many other ways to apply these transformations to images. For example it is possible to make modifications to these algorithms such that only a portion of the image is transformed. Multiple could also be applied in sequence or different transformations could be applied to different parts of an image, for example one could isolate different objects in an image and perform different transformations to them. The possibilities are endless all with the common purpose of adding artistic effects to images. 

\begin{figure}
	\centering
		\includegraphics[width=0.90\textwidth]{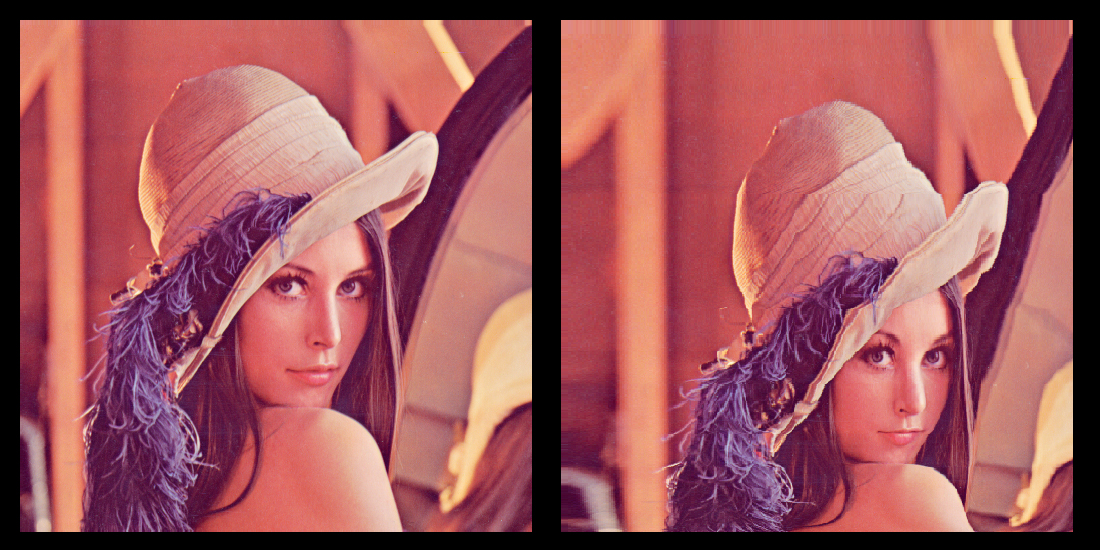}
	\caption{Lena before and after an affine fractal homeomorphism}
	\label{fig:Lena}
\end{figure}
\begin{figure}
	\centering
		\includegraphics[width=0.50\textwidth]{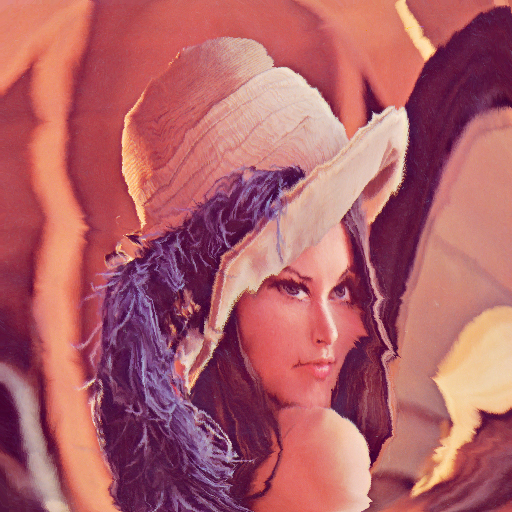}
	\caption{Lena before and after a bilinear fractal homeomorphism}
	\label{fig:Lena}
\end{figure}

\subsubsection{Video effects}

These transformations may also be used to create interesting video sequences. There are many ways in which this can be done. The first is to simply apply the transformations to individual frames of a video. By slightly changing the parameters of the transformations for each frame some nice effects can be produced. Starting with a still image it is possible to create anumations by using slightly different fractal transformations of the still to create each frame. This is able to generate complex types of motion, intricate swirling patterns and endless other effects. Even more complex animations can be generates by separating frames into layers and performing particular transformations to particular layers. 

Another way to produce interesting effects is to look at a video sequence as being a 3-dimensional object with each frame aligned with the x and y axes and the z axis being the time axis. By then constructing IFS's on the unit cube one can perform very complex fractal transformations to the video sequence data which can then be decompiled back into frames and played back. The result of such a transformation would be that points in the frames of the video sequence move around in both space and time. Provided that the transformation is continuous this can look really interesting. I will discuss more applications in 3D in the next section.

\subsubsection{Image synthesis}

Image synthesis can also be performed through the use of masked IFS's. This is described in \cite{bhi11:2011:TFI} and theoretically gives a way to synthesise two images into one and then extract them out just as they were. Of course, in practice it is not perfect but with the right parameters it can do a good job. One basic idea for this is that for two images you define a masked IFS over each such that the resulting address spaces have as little overlap as possible, that is, few codes in common. The two images are then mapped into the attractor of a third masked IFS whose address space contains the address spaces of the other two masked IFS's (or at least most of each). An easy way to choose the third IFS is simply to construct an IFS where the ranges of functions touch only on the boundaries, the tops code space of this IFS will then contain most of code space. Alternatively, by making the ranges of functions in the IFS completely disjoint then the tops code space is the entire code space, however this can be problematic since the resulting attractor is totally disconnected.

There are also methods to encode two images into one based on measure theory which also utilises fractal transformations. This will not be described here but one may instead refer to \cite{bhi11:2011:TFI}.

\subsubsection{Modelling}

Fractal transformations can also provide an alternative way of modelling rough objects. We explain in the context of an example, suppose we want to create a realistic looking image of the cross-section of a tree trunk. Taking concentric circles might be a good basic model but this is far from realistic looking in many cases. One approach may then be to just roughen each circle a little using traditional methods found in image rendering software. An alternative might be to perform a subtle fractal homeomorphism to an image of concentric circles. The result will roughen the edges in a way that can appear quite natural given a suitable choice of IFS's for the transformation. Of course there are many choices of IFS's which will do a poor job of this but with a little knowledge and experience making a good choice becomes easy and whole families of cross sections of tree trunks can be generated.

\begin{figure}[hbp]
	\centering
		\includegraphics[width=0.50\textwidth]{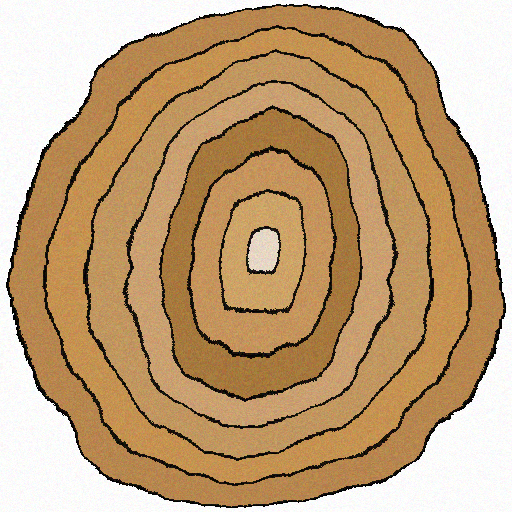}
	\caption{This tree trunk cross section was modelled using the process described above.}
	\label{fig:tree trunk cross section}
\end{figure}

The same method can also be applied to model rough 3-dimensional objects. For example we could model the terrain of a planet by starting with a sphere lying in the unit cube and then performing a fractal transformation to the unit cube in order to roughen the surface. Further, by using concentric spheres or perhaps a sphere with some internal gradient of data values, the resulting transformation will also roughen the internal structure so that you have a model of a planet and its core layers.

This demonstrates a major advantage of this method, not only can is it possible to get realistic models of the surface of rough 3-dimensional objects but it also has the potential to give nice models of the internal structure. Whilst it's possible to get very good results with this method there is a disadvantage that it can be quite difficult to get the IFS's right. It could take a lot of experimenting in some circumstances but it becomes a lot easier when you start to understand the way in which small changes in the IFS's will effect the transformation at all points.

\section{FRACTAL TRANSFORMATIONS IN 3D}

Now that we have seen many applications in 2-dimensions I will breifly describe how the algotithms can be extended to 3-dimensions. The theory previously described applies the same here as it did in 2-dimensions and the algorithms are simply extensions where functions on $\mathbb{R}^2$ are now functions in $\mathbb{R}^3$. The main difference here is that there are many more considerations to be made when dealing with 3D datasets. For simplicity we stick with affine IFS's defined within the unit cube $[0,1]^3$. One of the first things to consider when working in 3D is the type of data you wish to transform. This is an important question for many reasons. To begin with one should consider what the data is representing and how it visualised/viewed. This this may effect the suitable range of parameters of your transformation and you may need to think about how to best view the data after the transformation has been applied. Most importantly is the format of the data as this will effect the most efficient choice of algorithm. For example, there is a difference to how you may wish to handle volume data represented as voxels compared to surface mesh data represented as polygons.

Voxel data can be transformed using the perpixel algorithm, pixel chaining or by the chaos game. For affine IFS's both are able to accurately map the given data into a transformed data set in a comparable amount of time. In many ways voxel data is a lot easier to apply fractal transformations to than surface meshes. For polygon meshes it makes much more sense to apply a perpixel algorithm rather than the chaos game algorithm. This means for every vertex in the surface data you calculate the new vertex. This means that you do only one calculation per vertex and is more accurate than trying to apply the chaos game in this case as you can use the explicit vertex coordinates rather than wait for the chaos game to land within some error bound of each vertex. Another thing to keep in mind with polygon data is that fractal transformations do not necessarily preserve any planes. I.e. if you have an n-gon with n>3 in the data then it is generally assumed that this n-gon lies in a single plane. However this probably won't be the case after applying a fractal transformation. For this reason it is important that all polygons are reduced to triangles. There is a wide variety of triangulation schemes that may be used and they shall not be described here. By using triangles we ensure that the 3 vertices lie in a plane both before and after the fractal transformation. It is important to note that it may be posible for a triangle to be mapped to three points on the same line. If this were to occur then that triangle can be simply replaced with a line. Also if you have a mesh where at some sections it is described by large trangles whilst it is described by many small triangle in more detained sections then the fractal transformation can effect regions of small and large triangle in different ways. For this reason you may also want to retriangulate meshes such that they consist of triangles of roughly the same size, or within some sort of bound.

In either of these cases there is a range of 3-d visualisation software freely available online for viewing the data before and after a transformation. If you have a some programming experience it is also not too difficult to use tools like openGL for viewing such data, although viewing voxel data does take more effort than surface mesh data. As for perorming such a transformations there is currently no software readily available. However it is not too difficult to generate these transformations with high level mathematics doftware or even by writing your own program.

We provide an example consisting of 8 affine functions on the unit cube given by
\begin{equation} f_i(x,y,z)=\left( \begin{array}{l} s_1*x*t_{i,1}+((1-s_1)*x+s_1)*(1-t_{i,1}),\\
s_2*y*t_{i,2}+((1-s_2)*y+s_2)*(1-t_{i,2}),\\
s_3*z*t_{i,3}+((1-s_3)*z+s_3)*(1-t_{i,3}) \end{array} \right) \end{equation}
where for $i\in\{1,8\}$ and each $\tilde{t_i}=(t_{i,1},t_{i,2},t_{i,3})$ is a unique vector pointing at one of the 8 vertices of the unit cube $[0,1]^3$, and $\tilde{s}=(s_1,s_2,s_3)\in(0,1)^3$. The IFS given by these functions is always a just touching IFS and has the unit cube as the unique attractor. Given any two choices of $\tilde{s}\in(0,1)^3$ it is possible to construct a fractal homeomorphism between the two attractors. Applying these transformations to objects and data sets is as simple as embedding this data into the unit cube, then applying a fractal transformation to map the data embedded in this cube into another cube which is the attractor of the above IFS for some $\tilde{s}'$.
We provide some images of fractal transformations after having been applied to voxel data sets and surface meshes. Some of these transformations have been constructed using the IFS family above while others are from more complex IFS's consisting of mappings that are affine on their edges.


\begin{figure}
	\centering
		\includegraphics[width=0.60\textwidth]{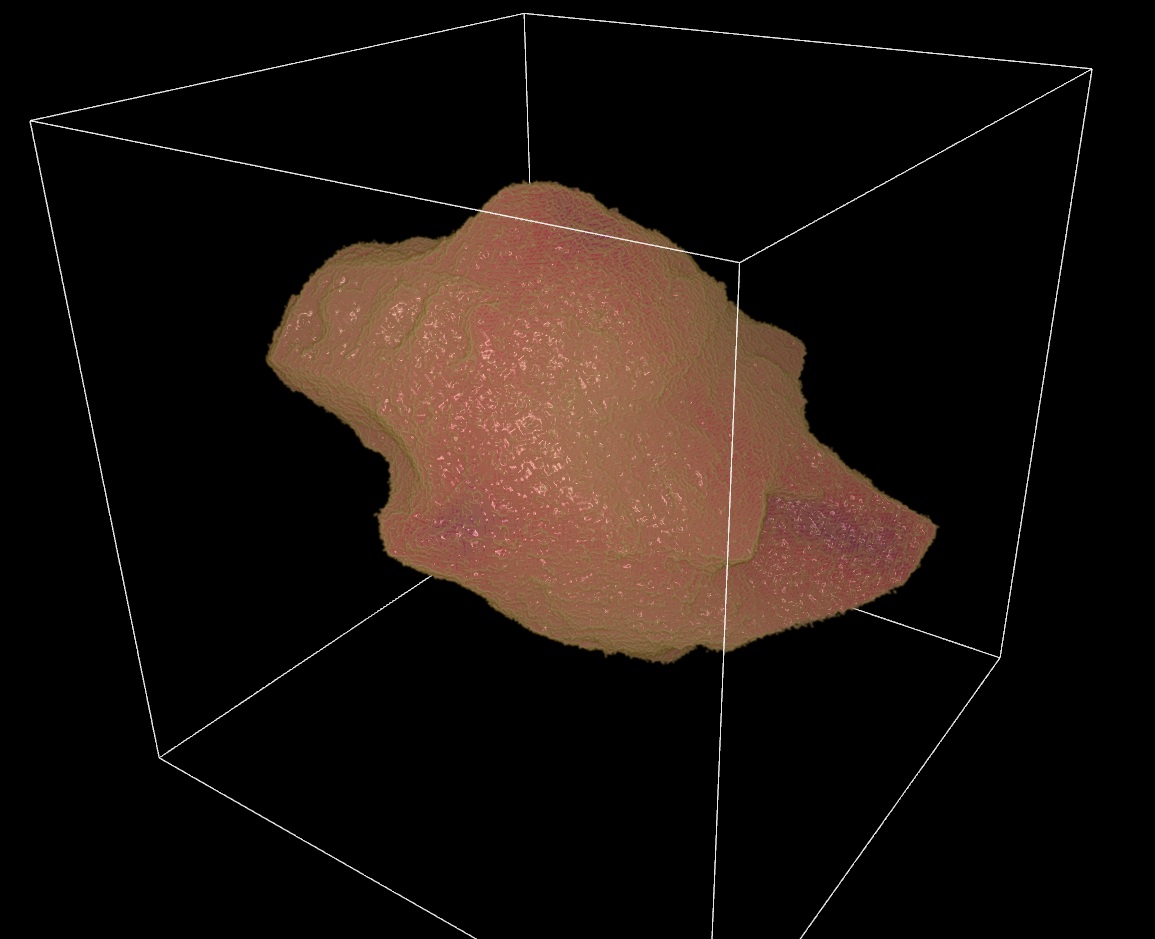}
	\caption{This image is a rendering of a solid sphere after a fractal transformation. The data here is represented as voxels and is viewed using Drishti.}
	\label{fig:homeomorphism of a sphere}
\end{figure}
\begin{figure}
	\centering
		\includegraphics[width=0.60\textwidth]{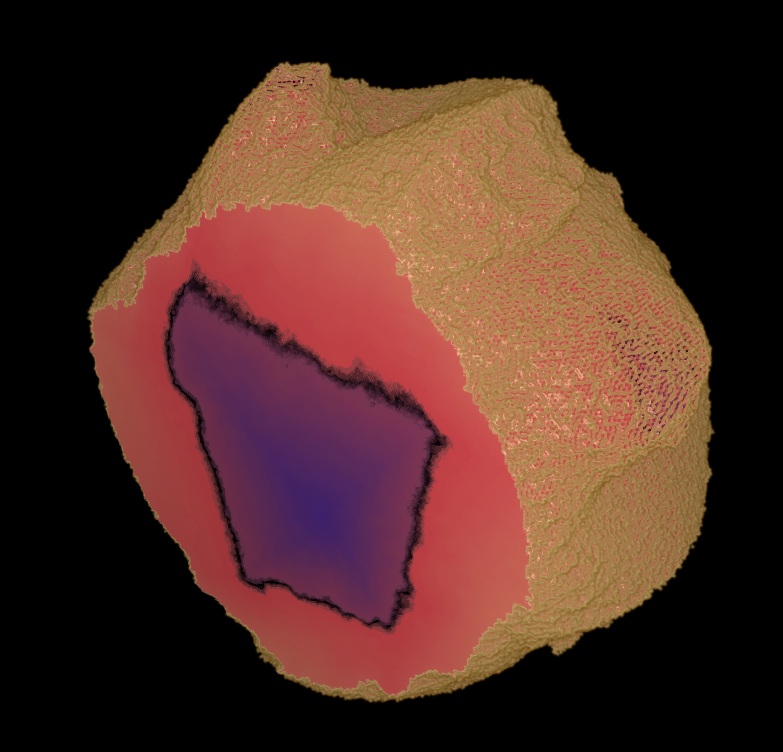}
	\caption{This is a different homeomorphism of a solid sphere. Here we also take a look at a cross section. Not the internal structure of the sphere consisted of concentric spheres, the dark quasi-circle shows a cross section of one of these layers after the transformation.}
	\label{fig:homeomorphism of a sphere}
\end{figure}
\begin{figure}
	\centering
		\includegraphics[width=0.90\textwidth]{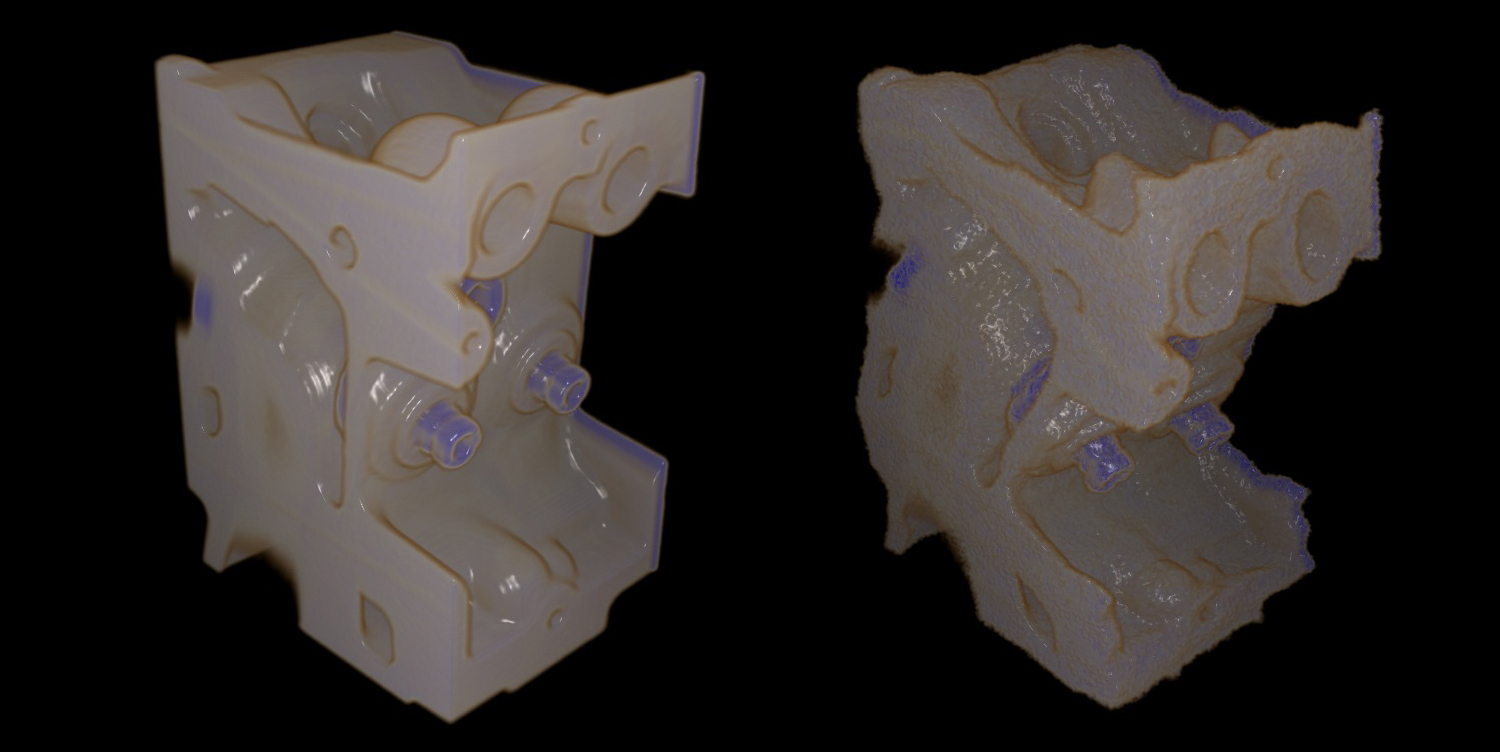}
	\caption{Before and after a fractal homeomorphism}
	\label{fig:engine block}
\end{figure}
\begin{figure}
	\centering
		\includegraphics[width=0.90\textwidth]{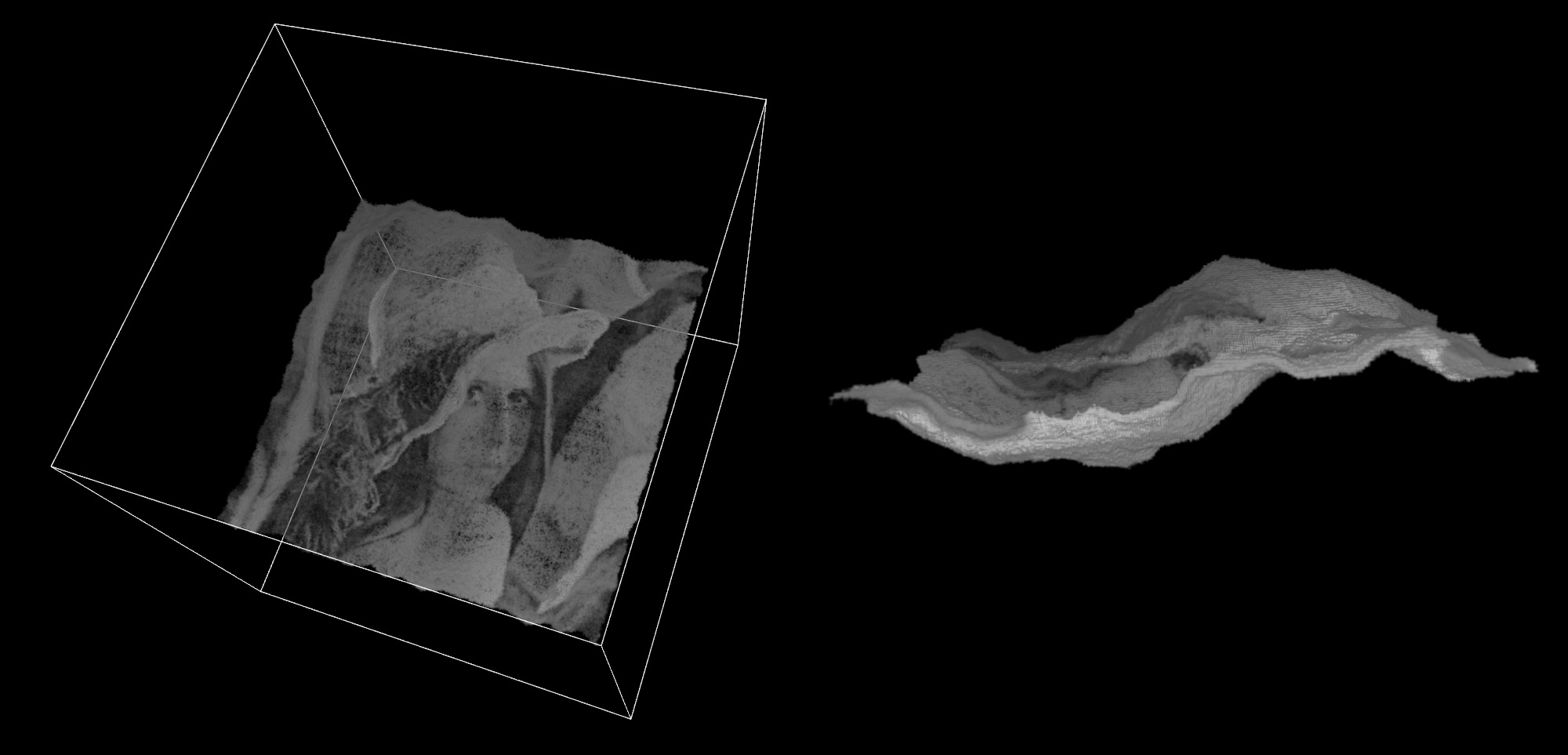}
	\caption{This crumpled photograph of Lena was modelled by applying a fractal homeomorphism to a flat photograph. Note that the second ange looks like is could be a terrain model.}
	\label{fig:photograph of Lena}
\end{figure}
\begin{figure}
	\centering
		\includegraphics[width=0.90\textwidth]{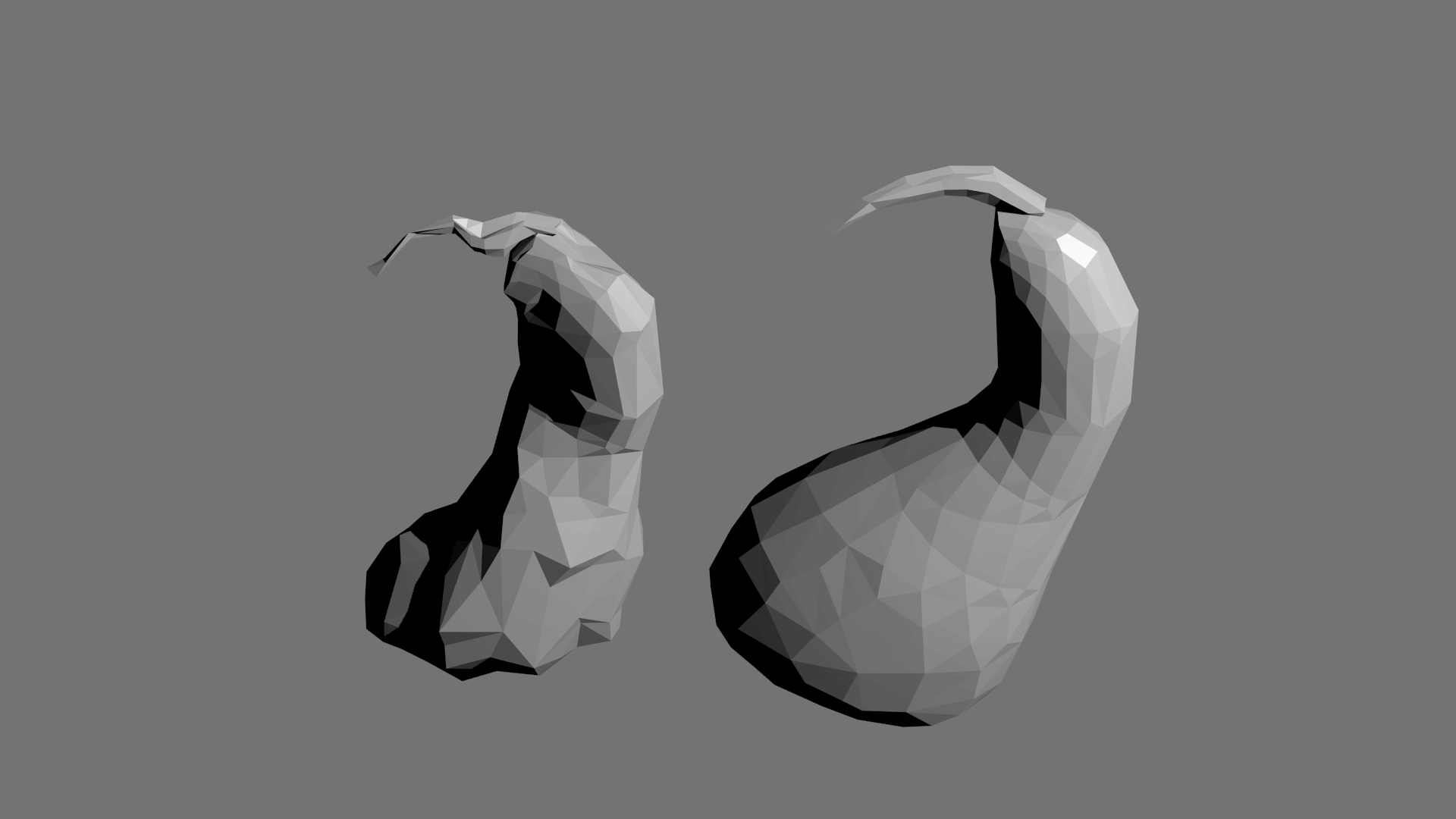}
	\caption{This is a simple rendering of a fractal homeomorphism applied to the mesh model of a gourd. On the right is the before and on the left is after the transformation is applied.}
	\label{fig:gourd mesh model}
\end{figure}

\section{CONCLUSION} 

Fractal transformations have a lot of potential in the computer graphics industry as seen in past and present uses, although there is still much to be done in developing useful applications for fractal transformations. There is also a lot to be done with the developement of algorithms, in particular decreasing the runtime for large datasets with IFS's consisting of complicated functions. Developing efficient paralell algorithms is also be important with the goal of producing real-time interactive fractal transformations and effects. User friendly and flexible algorithms that are able to dynamically choose the best algorithm for applying transformations are also things to look forward to in the near future.

\section*{ACKNOWLEDGEMENTS}

To someone.
Thanks

\bibliographystyle{ws-acs}
\bibliography{template}


\end{document}